\documentclass[a4paper,12pt,reqno]{amsart}
 
\usepackage[margin=2cm]{geometry} 
\usepackage{amsmath,amsthm,amssymb,enumitem}
\usepackage{amsrefs}
\usepackage[ocgcolorlinks,hyperfootnotes=false,colorlinks=true,citecolor=blue,linkcolor=blue,urlcolor=blue]{hyperref}

\newcommand{\fl}{f{\kern0.075em}l}
\newcommand{\norm}[1]{\left\lVert #1 \right\rVert}

\theoremstyle{plain}
\newtheorem{thm}{Theorem}[section]
\newtheorem{lem}[thm]{Lemma}
\newtheorem{prop}[thm]{Proposition}

\theoremstyle{definition}
\newtheorem{defn}[thm]{Definition}
\newtheorem{exmp}[thm]{Example}

\setlist[enumerate]{itemsep=2mm, topsep=0mm}

\title[Asymptotic behaviour of the Bergman kernel and metric]{Asymptotic behaviour of the Bergman kernel and metric
}
\author{Ravi Shankar Jaiswal}
\address{Centre for Applicable Mathematics, Tata Institute of Fundamental Research, Bangalore 560065, India.}
\email{ravi@tifrbng.res.in}
\date{November 2, 2023}

\begin{document}

\addtolength{\jot}{2mm}
\addtolength{\abovedisplayskip}{1mm}
\addtolength{\belowdisplayskip}{1mm}

\maketitle
\begin{abstract}
    We prove nontangential asymptotic limits of the Bergman kernel on the diagonal, and the Bergman metric and its holomorphic sectional curvature at exponentially \fl{}at
 infinite type boundary points of smooth bounded pseudoconvex
 domains in $\mathbb{C}^{n + 1}$, $n \in \mathbb{N}$. We first show that these objects satisfy appropriate localizations and then use the method of scaling to complete the proof. 
\end{abstract}
\section{Introduction}\label{Intro} 
The primary goal of this article is to describe the asymptotic behaviour of the Bergman kernel on the diagonal, and the Bergman metric and its holomorphic sectional curvature at an infinite type boundary point of a smooth bounded pseudoconvex domain in $\mathbb{C}^{n+1}$, $n \in \mathbb{N}$.
We focus on exponentially flat infinite type boundary points of a smooth bounded pseudoconvex domain. Exponentially flat infinite type points are generalizations of the origin in  $\operatorname{Re}z_1+ e^{-1/|z'|^2} = 0$, $(z_1, z') \in \mathbb{C} \times \mathbb{C}^n$.

Let $D \subset \mathbb{C}^{n + 1}$ be a bounded smooth domain with $0 \in bD$. The boundary point
$0$ is said to be \emph{exponentially flat} if there exists a local defining function of D near the origin of the form
\begin{equation}\label{1}
    \rho(z_1, \dots, z_{n + 1}) = \operatorname{Re}z_1 + \phi \left(|z_2|^2 + \dots + |z_{n + 1}|^2\right),
\end{equation}
 where $\phi : \mathbb{R} \to \mathbb{R}$ is a smooth function satisfying the Definition \ref{exp flat fun}.
\begin{defn}\label{exp flat fun}
    A smooth function $\phi: \mathbb{R} \to \mathbb{R}$ is said to be
    \emph{exponentially flat near the origin}, if it satisfies the following properties:
    \begin{enumerate}
    \item $\phi(x) = 0,$ for $x \leq 0$, 
    \item $\phi^{(n)}(0) = 0$, for $n \in \mathbb{N}$,
    \item $\phi''(x) > 0$, for $x > 0$,
    \item the function $-1/\operatorname{log}\phi(x)$ extends to a smooth function on $[0,\infty)$ and vanishes to a finite order $m$ at $0$, and
    \item \begin{equation}\label{Scaling}
        \lim_{r \to 0^{+}}\frac{\phi(rx)}{\phi(r)} = \begin{cases} 
      0,  &\text{if } 0 < x < 1, \text{ and}\\
      \infty,  & \text{if } x > 1.
      \end{cases}
    \end{equation}
\end{enumerate}
\end{defn}
\begin{exmp}
For $m \in \mathbb{N}$, $\phi(x) = \operatorname{exp}\left({-1}/{x^m}\right)$ is exponentially flat near the origin.
\end{exmp}
We now introduce the necessary definitions and notations that are
needed to state our results. 
We use the following generalized cones for nontangential approach. 
For $\alpha, N > 0$, let
\begin{equation}
    C_{\alpha,N} := \Big\{\big(z_1,z'\big) \in \mathbb{C}\times\mathbb{C}^{n} : \operatorname{Re}z_1 < -\alpha|z'|^N\Big\},
\end{equation}
where $z' = (z_2, \dots, z_{n + 1})$.
An \emph{$(\alpha, N)$-cone type stream} approaching $0 \in bD$ is a smooth curve $q : (0, \epsilon_0) \to D \cap C_{\alpha, N}$ with
\begin{equation*}
    \lim_{t \to 0^+} q(t) = (0, \dots, 0),
\end{equation*}
for some $\epsilon_0 > 0$.
{Let $\xi \in \mathbb{C}^{n + 1}$. We use $\xi_{N,t}$ and $\xi_{T,t}$ to denote the complex normal and complex tangential components of $\xi$ with respect to $\pi(q(t))$, where $\pi\colon \mathbb{C}^n \to bD$ is the normal projection onto $bD$ for points close to $bD$.
So, $\xi = \xi_{N,t} + \xi_{T,t}$ with $\xi_{T,t} \in T_{\pi(q(t))}^{\mathbb{C}}(bD)$ and $\xi_{N,t} \perp T_{\pi(q(t))}^{\mathbb{C}}(bD)$. 
Let $d(t)$ and $d^*(t)$ represent the Euclidean distance of $q(t)$ to $bD$ along the normal and tangential directions.}

Boas, Straube, and Yu \cite{Yu 1995} studied the nontangential asymptotic behaviour of the Bergman kernel, metric and its holomorphic sectional curvature for a wide class of weakly pseudoconvex bounded domains in $\mathbb{C}^{n+1}$ of finite type (in the sense of D'Angelo \cite{D'Angelo 1982}).
In this paper, we extend the results of Boas, Straube, and Yu \cite{Yu 1995} to a class of infinite type pseudoconvex domains in $\mathbb{C}^{n + 1}$. We recall a special case of their result below in Theorem \ref{1.1}.
\begin{thm}\cite{Yu 1995}*{Theorems 1 and 2}\label{1.1}
    Let $D$ be a bounded pseudoconvex domain in $\mathbb{C}^{n + 1}$ with $0 \in bD$. Suppose the local defining function of $D$ near the origin is given by
    \[\rho(z_1, \dots, z_{n + 1}) = \operatorname{Re}z_1 + |z'|^{2m},\]
    where $z' = (z_2, \dots, z_{n + 1}) \in \mathbb{C}^{n}$ and $m \in \mathbb{N}$. Then, for $\xi = (\xi_1, \xi') \in \mathbb{C}^{n + 1} \setminus \{0\}$, there exist constants $C_{1,m}, C_{2,m} > 0$, and $C_{3,m}$ such that
    \begin{enumerate} [labelsep=10mm,leftmargin=20mm,itemsep=5mm]
        \item 
        $\begin{aligned}[t]
             \lim_{t \to 0^+} \frac{{\kappa}_D(-t, 0)}{d(t)^{-2}d^*(t)^{-2n}} = C_{1,m}, 
        \end{aligned}$
    \item 
    $\begin{aligned}[t]
     \lim_{t \to 0^+} \frac{B_D((-t, 0); \xi)}{\sqrt{{|\xi_1|^2}/{d(t)^2} + {|\xi'|^2}/{d^*(t)^2}}} = C_{2, m},   \text{ and} 
    \end{aligned}$
    \item 
    $\begin{aligned}[t]
        \lim_{t \to 0^{+}} H_D((-t, 0); \xi) = C_{3, m}.
    \end{aligned}$
    \end{enumerate}
   Here, $\kappa_D$ denotes the Bergman kernel on the diagonal of $D$, $B_D$ denotes the Bergman metric of $D$, 
    and $H_D$ denotes the holomorphic sectional curvature of the Bergman metric of $D$. In this setting $d(t) = t$ and $d^{*}(t) = t^{1/2m}$.
\end{thm}
The main result of this article is the following.
\begin{thm}\label{1.3}
    Let $D \subset \mathbb{C}^{n + 1}$ be a smooth bounded pseudoconvex domain with $0 \in bD$. If $0$ is an exponentially flat boundary point and $q$ is an $(\alpha, N)$-cone type stream approaching $0$, then, for $\xi \in \mathbb{C}^{n + 1} \setminus \{0\}$, 
    \begin{enumerate}[labelsep=10mm,leftmargin=20mm,itemsep=5mm]
    \item $\begin{aligned}[t]
        \lim_{t \to 0^+}\frac{\kappa_D(q(t))}{d(t)^{-2}d^*(t)^{-2n}} = \frac{1}{4\pi\operatorname{vol}(B_n(0,1))},
    \end{aligned}$
        \item $\begin{aligned}[t]
     \lim_{t \to 0^+} \frac{B_D(q(t); \xi)}{\sqrt{{|\xi_{N,t}|^2}/{2d(t)^2} + (n+1) {|\xi_{T,t}|^2}/{d^*(t)^2}}} = 1,   \text{ and} 
    \end{aligned}$
        \item $\begin{aligned}[t]
        \lim_{t \to 0^{+}} H_D(q(t); \xi) \left(\frac{{2|\xi_{N, t}/2d(t)|^4 + (n + 1)|\xi_{T,t}/d^{*}(t)|^4}}{\left(2|\xi_{N, t}/2d(t)|^2 + (n + 1)|\xi_{T, t}/d^*(t)|^2\right)^2}\right)^{-1} = -2.
    \end{aligned}$
    \end{enumerate} 
\end{thm}
Some of our results above, in $n = 1$, are stated in Kim and Lee \cite{Sunhong 2002}. However, \cite{Sunhong 2002}*{Proposition 2}
is not true and it is used crucially in their proofs. We explicitly highlight this error in Section \ref{counterexample}. Our results provide an alternate proof for some of the results in Kim and Lee \cite{Sunhong 2002}.

The asymptotic behaviour of the Bergman kernel and metric has been the subject of profound interest and investigation in complex analysis.
Bergman \cites{Bergman 1970, Stefan 1933} studied the boundary behaviour of the Bergman kernel on certain special classes of domains in $\mathbb{C}^2$ (strongly pseudoconvex domains, domains with certain star symmetries, etc).
Hörmander \cite{Hormander 1965} extended the analysis to bounded strongly pseudoconvex domains in $\mathbb{C}^{n+1}$ (see \cite{Hormander 1965}*{Theorem 3.5.1}).
Kamimoto \cite{Kamimoto} has described the asymptotic behaviour of the Bergman kernel at the boundary for some pseudoconvex model domains (including not only finite type cases but also some infinite type cases) using
the geometrical information of the Newton polyhedron of the
defining function.

Diederich's \cites{Diederich 1970, Diederich 1973} contributions were focused on the asymptotic behaviour of the Bergman metric for strongly pseudoconvex domains in $\mathbb{C}^{n+1}$.
Klembeck \cite{Klembeck 1978} leveraged asymptotic results for the Bergman kernel function of Fefferman \cite{Fefferman 1974} to derive the asymptotic behaviour of the holomorphic sectional curvature of the Bergman metric in strongly pseudoconvex domains. 

There has been significant progress in providing lower and upper bounds of the Bergman kernel and metric on weakly pseudoconvex domains of finite type and some classes of infinite type domains.
For results of this nature, see
Catlin \cite{Catlin}, Nagel-Rosay-Stein-Wainger \cite{Nagel}, McNeal \cites{McNeal 1989, McNeal 1992, McNeal 1994}, Herbort \cites{Herbort 1992, Herbort 1993}, Diedrich \cites{Diedrich 1993, Diedrich 1994}, Bharali \cites{Bharali 2010, Bharali 2020}.

The article is organized as follows. Section \ref{Prelim} is devoted to definitions, known results, and preliminary findings essential for proving our main results. 
In Section \ref{Cone}, we provide the proofs of our main results.
Finally, in Section \ref{counterexample}, we address the errors in Kim and Lee {\cite{Sunhong 2002}}.
\section{Preliminaries}\label{Prelim}
Let $D \subset \mathbb{C}^{n + 1}$ be a bounded domain. The Bergman space of $D$, denoted by $A^2(D)$, consists of holomorphic functions $f:D\to\mathbb{C}$ such that $\int_D|f|^2\,\mathrm{d}V<\infty$. Since $A^2(D)$ is a closed subspace of $L^2(D)$, there exists an orthogonal projection of $L^2(D)$ onto $A^2(D)$, called the Bergman projection of $D$. The Bergman projection, denoted by $P_{D}$, is given by an integral kernel, called the Bergman kernel and denoted by $K_D: D\times D\to\mathbb{C}$, i.e.
\begin{equation}
    P_{D}f(z) = \int_{D}K_{D}(z, w)f(w) \, \mathrm{d}V(w) \quad\text{for all } f \in L^2(D), z \in D.
\end{equation}

The Bergman kernel satisfies the following properties: 
\begin{enumerate}
    \item For a fixed $w\in D$, $K_D(\cdot, w)\in A^2(D)$,
    \item $\overline{K_D(z, w)} = K_D(w,z)$ for all $z,w\in D$,and
    \item $f(z)=\int_D K_D(z, w)f(w)\,\mathrm{d}V(w)$ for all $f\in A^2(D)$ and $z\in D$.
\end{enumerate} 

The Bergman kernel on the diagonal of $D$ is denoted by $\kappa_D(z) = K_D(z,z)$. Moreover, we can define the Bergman metric of $D$ as 
\begin{equation*} 
B_D(z;\xi)=\sqrt{\sum_{i,j=1}^{n+1}g_{i\bar{j}}(z)\xi_i\overline{\xi}_j},
\end{equation*} where 
$g_{i\bar{j}}(z) = \partial^2_{z_i\bar{z}_j} \operatorname{log} \kappa_D(z)$, $z \in D$, and $\xi = (\xi_1, \xi_2, \dots, \xi_{n+1}) \in \mathbb{C}^{n+1}$.

The holomorphic sectional curvature for the Bergman metric is given by
\begin{equation*}
    H_{D}(z; \xi) = \frac{\sum_{h,j,k,l}R_{\bar{h}jk\bar{l}}(z)\bar{\xi}_h \xi_j \xi_k \bar{\xi}_l}{B_D(z;\xi)^4}, \, z \in D, \, \text{and } \xi \in \mathbb{C}^{n + 1} \setminus \{0\}.
\end{equation*}
In this equation, $R_{\bar{h}jk\bar{l}}(z) = -\partial^2_{z_k\bar{z}_l}g_{j\bar{h}}(z) + \sum_{\mu, \nu} g^{\nu \bar{\mu}}(z) \partial_{z_k}g_{j\bar{\mu}}(z) \partial_{\bar{z}_l}g_{\nu \bar{h}}(z) ,$ where $[g^{\nu \bar{\mu}}(z)]$ is the inverse matrix of $[g_{j \bar{k}}(z)]$.

We define the following extremal integrals that are useful in computing and estimating $k_D, \, B_D$, and $H_D$. Let $z \in D$, $\xi \in \mathbb{C}^{n + 1} \setminus \{0\}$. Define
\begin{align}
    I_0^D(z) &= \operatorname{inf}\left\{\int_D|f|^2 \, \mathrm{d}V : f \in A^2(D), f(z) = 1\right\},\\
    I_1^D(z, \xi) &= \operatorname{inf}\left\{\int_D|f|^2 \, \mathrm{d}V: f \in A^2(D), f(z) = 0, \sum_{j= 1}^{n+1}\xi_j\frac{\partial f}{\partial z_j}(z) = 1\right\} \text{, and}\\
    I_2^D(z, \xi) &= \operatorname{inf}\Bigg\{\int_D|f|^2 \, \mathrm{d}V: f \in A^2(D),  f(z) = \frac{\partial f}{\partial z_j}(z) = 0, j = 1, \dots, n+1, \nonumber\\[-2mm]
    & \hspace{80mm} \text{and}\sum_{k, l= 1}^{n+1}\xi_k \xi_l\frac{\partial^2 f}{\partial z_k \partial z_l}(z) = 1\Bigg\}.
\end{align}
It is easy to see from the definitions that if $D'$ is a subdomain of $D$ and $z \in D'$, then
\[I_0^{D'}(z) \leq I_0^{D}(z), \quad I_1^{D'}(z, \xi) \leq I_1^{D}(z, \xi), \text{ and}\quad I_2^{D'}(z; \xi, \eta) \leq I_2^D(z; \xi, \eta).\]

If $f: D_1 \to D_2$ is a biholomorphism. Then the following transformation formulae hold.
\begin{align}
        I_0^{D_1}(z)|\operatorname{det}J_{\mathbb{C}}f(z)|^2 &= I_0^{D_2}(f(z)),\\
        I_1^{D_1}(z, \xi)|\operatorname{det}J_{\mathbb{C}}f(z)|^2 &= I_1^{D_2}(f(z), J_{\mathbb{C}}f(z) \xi), \text{ and}\\
        I_2^{D_1}(z , \xi)|\operatorname{det}J_{\mathbb{C}}f(z)|^2 &= I_2^{D_2}(f(z), J_{\mathbb{C}}f(z) \xi),
\end{align}
where $J_{\mathbb{C}}f(p)$ denotes the complex Jacobian matrix of $f$ at $p$. 

The following proposition gives a useful representation of the Bergman kernel on the diagonal, and the Bergman metric and its holomorphic sectional curvature in terms of the above extremal integrals.
\begin{prop}\label{Fuchs}(Bergman-Fuchs \cite{Stefan 1933})
    Let $D \subset \mathbb{C}^{n + 1}$ be a domain. Then, for $z \in D$
    and $\xi \in \mathbb{C}^{n + 1} \setminus \{0\}$
\begin{align*}
        \kappa_D(z) &= \frac{1}{I_0^D(z)}, \quad
        B_D^2(z; \xi) = \frac{I_0^D(z)}{I_1^D(z; \xi)}, \text{ and } \quad H_D(z; \xi) = 2 - \frac{\left(I_1^D(z, \xi)\right)^2}{I_0^D(z) \cdot I_2^D(z, \xi)}.  
\end{align*}
We will need the following formulae for the Bergman kernel, the Bergman metric, and its holomorphic sectional curvature to prove our results.
\begin{lem}\label{Formula for kernel} The Bergman kernel, and the Bergman metric and its holomorphic sectional curvature at $ 0 \in \mathbb{D} \times B_n(0, 1)$ are as follows.
\begin{align*}
    \kappa_{\mathbb{D} \times B_n(0, 1)}(0) &= \frac{n!}{\pi^{n + 1}}, \\
    B_{\mathbb{D} \times B_n(0, 1)}(0; \xi) &= \sqrt{2|\xi_1|^2 + (n + 1)|\xi'|^2}, \, \text{and} \\ H_{\mathbb{D} \times B_n(0, 1)}(0; \xi) &= 
 -\frac{2\left(2|\xi_1|^4 + (n + 1)|\xi'|^4\right)}{\left(2|\xi_1|^2 + (n + 1)|\xi'|^2\right)^2},
\end{align*}
for $\xi \in \mathbb{C}^{n + 1} \setminus \{0\}$.
\end{lem}
\begin{proof}
    From \cite{Jarnicki}*{Example $6.1.5$ and Theorem $6.1.11$}, we get
    \begin{align*}
   K_{\mathbb{D} \times B_n(0, 1)}((z_1, z'),(w_1, w')) &= K_{\mathbb{D}}(z_1, w_1) \cdot K_{B_n(0, 1)}(z', w')\\
   &= \frac{n!}{\pi^{n + 1}}(1 - z_1\bar{w}_1)^{-2}\left(1 - \sum_{j = 2}^{n + 1}z_j\bar{w}_j\right)^{-(n + 1)},     \end{align*}
    for all $z_1, w_1 \in \mathbb{D}$ and $z', w' \in B_n(0, 1)$. 
    
    Using the
    above formula for the Bergman kernel, we obtain the Bergman metric for the product domain $\mathbb{D} \times B_n(0,1)$.
    \begin{equation*}
        g_{j\bar{k}}(z_1, z') = \begin{cases} 
      \frac{2}{\left(1 - |z|^2\right)^2}, &\text{if } j = k = 1,\\[2mm]
      \frac{n  + 1}{\left(1 - |z'|^2\right)^2}\left[\left(1 - |z'|^2\right)\delta_{jk} + \bar{z}_j z_k\right],&\text{if } j, k \in \{2, \dots, n + 1\},\\[2mm]
      0, &\text{if } j = 1, k \in \{2, \dots, n + 1\}, \text{ and}\\[2mm]
      0, &\text{if } k = 1, j \in \{2, \dots, n + 1\}.
   \end{cases}
    \end{equation*}
    Therefore,
    \begin{equation*}
        B_{\mathbb{D} \times B_n(0, 1)}(0; \xi) = \sqrt{2|\xi_1|^2 + (n + 1)|\xi'|^2},
    \end{equation*}
    for $\xi \in \mathbb{C}^{n + 1}$.
    
    To compute the holomorphic sectional curvature at the origin for the Bergman metric of $\mathbb{D} \times B_n(0, 1)$,
    we consider
    \begin{align*}
    \sum_{h, j,k,l}R_{\bar{h}jk\bar{l}}(0)\bar{\xi}_h\xi_j\xi_k\bar{\xi}_l &= \sum_{h,j,k,l}\left(- \frac{\partial^2 g_{j\bar{h}}}{\partial z_k \partial \bar{z}_l}(0) + \sum_{\mu, \nu} g^{\nu \bar{\mu}}(0) \frac{\partial g_{j \bar{\mu}}}{\partial z_k}(0) \frac{\partial g_{\nu \bar{h}}}{\partial \bar{z}_l}(0)\right)\bar{\xi}_h\xi_j\xi_k\bar{\xi}_l\\
    & = - 2\left(2|\xi_1|^4 + (n + 1)|\xi'|^4\right).
    \end{align*}
    Hence,
    \begin{equation*}
        H_{\mathbb{D} \times B_n(0, 1)}(0; \xi) = -\frac{2\left(2|\xi_1|^4 + (n + 1)|\xi'|^4\right)}{\left(2|\xi_1|^2 + (n + 1)|\xi'|^2\right)^2},
    \end{equation*}
    for $\xi \in \mathbb{C}^{n + 1} \setminus \{0\}$.
\end{proof}
\end{prop}
We will prove localization for the extremal integrals (Lemma \ref{localisation lemma 4}) by utilizing the following theorem. To state the theorem, we need the following definition. For a plurisubharmonic function $\phi$ on $D$, let
\[L^2_{(0, 1)}(D, \phi) := \left\{g = \sum_{j = 1}^{n}g_j \mathrm{d}\bar{z}_j : \int_{\Omega} |g|^2\operatorname{e}^{-\phi} \, \mathrm{d}V = \int_{\Omega} \left(\sum_{j = 1}^n |g_j|^2\right) \operatorname{e}^{-\phi} \, \mathrm{d}V < \infty \right\}.\]
\begin{thm}(Hörmander \cite{Hörmander})\label{Hormander} Let $D$ be a pseudoconvex open set in $\mathbb{C}^n$ and $\phi$ any plurisubharmonic function in $D$. For every $g \in L^2_{(0, 1)}(D, \phi)$ with $\overline{\partial}g = 0$, there is a solution $u \in L^2_{\operatorname{loc}}(D)$ of the equation $\overline{\partial}u = g$ such that
\begin{equation*}
    \int_{D} \frac{|u|^2e^{-\phi}}{\left(1 + |z|^2\right)^{2}}\, \mathrm{d}V \leq \int_D|g|^2e^{-\phi} \, \mathrm{d}V.
\end{equation*}
\end{thm}
\section{Main Results}\label{Cone}
In this section, we will prove the nontangential asymptotic limits of the Bergman kernel on the diagonal, and the Bergman metric and its holomorphic sectional curvature at exponentially flat infinite type boundary points of bounded smooth pseudoconvex domains in $\mathbb{C}^{n + 1}$.

Let $D \subset \mathbb{C}^{n + 1}$ be a bounded smooth domain with $0 \in bD$ and $q(t)$ be an $(\alpha, N)$ cone type stream approaching $0$. If $0$ is an exponentially flat boundary point, then there exists a neighbourhood $U$ of origin, such that
\begin{equation}\label{*}
    D \cap U = \{z \in U: \rho(z) < 0\},
\end{equation}
where the defining function $\rho$ is defined as in \eqref{1}. 

Define $T^1_t : \mathbb{C}^{n+1} \to \mathbb{C}^{n+1}$ by
\begin{equation}
    T^1_t(w) = (w_1 - i\operatorname{Im}q_1(t), w_2, \dots, w_{n+1}),
\end{equation}
and a Hermitian map $R_t^1 : \mathbb{C}^{n + 1}  \to \mathbb{C}^{n + 1}$ by
\begin{equation}
R_t^1(x) = \begin{cases}
x, & \text{if } r(t) := \| (0, q_2(t), \dots, q_{n+1}(t)) \| = 0,\\[2mm]
e_1, & \text{if } x = e_1 \text{ and } r(t) \neq 0, \text{ and}\\[2mm]
e_2, & \text{if } x = \frac{(0, q_2(t), \dots, q_{n+1}(t))}{\| (0, q_2(t), \dots, q_{n+1}(t)) \|} \text{ and } r(t) \neq 0.
\end{cases}
\end{equation}
Now, we have $\Tilde{q}(t) := R^1_t\circ T^1_t(q(t)) = (\operatorname{Re}q_1(t), r(t), 0,\dots, 0) $, with $\operatorname{Re}q_1(t) < 0$.
Let $\delta_0 > 0$ such that 
\begin{equation}
    (-4\delta_0, 4\delta_0)^2 \times B_n(0, 4\delta_0) \subset U.
\end{equation} 
Since $q(t) \to 0$ as $t \to 0^{+}$, $q(t) \in (-\delta_0, \delta_0)^2 \times B_n(0, \delta_0) \subset U$ for all sufficiently small $t > 0$.  

Hence, $\Tilde{q}(t) \in \Omega \cap R^1_t\left(T^1_t \left((-\delta_0, \delta_0)^2 \times B_n(0, \delta_0)\right)\right) \cap C_{\alpha, N} \subset \Omega \cap U \cap C_{\alpha,N}$ for all sufficiently small $t > 0$, where $\Omega = \big\{z \in \mathbb{C}^{n+1} : \rho(z) < 0\big\}$.

For small $t > 0$, there exists a unique point $p(t) \in b\Omega$
such that $\operatorname{dist}(\Tilde{q}(t), b\Omega) = |\Tilde{q}(t) - p(t)|$, $p_1(t) \in \mathbb{R}$, $p_2(t) \geq 0$, and $p_j(t) = 0$, for all $j = 3, \dots, n+1$. Therefore, $p_1(t) = - \phi\left(p_2(t)^2\right) \leq 0$. Since $p(t) \to 0$ as $t \to 0^{+}$,
$p(t) \in B_{n+1}(0, \delta_0)$ for all sufficiently small $ t > 0$.

Since $\rho(z) = \operatorname{Re}z_1 + \phi\left(|z_2|^2 + \dots + |z_{n + 1}|^2\right)$, $\nabla\rho\left(p(t)\right) = \left(1, 2p_2(t)\phi'(p_2(t)^2), 0\right) $.

Now define $T^2_t : \mathbb{C}^{n+1} \to \mathbb{C}^{n+1}$ such that
\begin{equation}
    T^2_t(w) = w - p(t),
\end{equation}
and a Hermitian linear map $R^2_t : \mathbb{C}^{n+1} \to \mathbb{C}^{n+1}$, which maps 
\begin{equation}
    \frac{\nabla \rho (p(t))}{\norm{\nabla \rho(p(t))}} \mapsto e_1, \frac{(-2p_2(t)\phi'(p_2(t)^2), 1, 0)}{\norm{\nabla \rho(p(t))}} \mapsto e_2, \text{ and }e_j \mapsto e_j, \text{ for } j = 3, \dots, n+1.  
\end{equation}
Define $\gamma_t = R^2_t \circ T^2_t $ and $A(t) = \|\nabla \rho(p(t))\|$, therefore
\begin{align}
\left(\gamma_t\right)^{-1}(z_1, \dots, z_{n+1}) &= \bigg(\frac{z_1 - 2p_2(t)z_2\phi'(p_2(t)^2)}{A(t)} - \phi(p_2(t)^2), \frac{z_2 + 2z_1p_2(t)\phi'(p_2(t)^2)}{A(t)} \nonumber\\
& \hspace{80mm} + p_2(t), z_3, \dots, z_{n+1}\bigg), \\
\rho \circ \gamma_t^{-1}(z_1, \dots, z_{n+1}) &= \frac{\operatorname{Re}z_1}{A(t)} - \frac{2p_2(t)\phi'(p_2(t)^2)}{A(t)}\operatorname{Re}z_2 - \phi(p_2(t)^2) \nonumber\\
& \hspace{2mm}  + \phi\left(\left|\frac{z_2}{A(t)} + p_2(t) + \frac{2p_2(t)\phi'(p_2(t)^2)}{A(t)}z_1\right|^2 + |z_3|^2 + \dots + |z_{n+1}|^2\right).
\end{align}
Since $W := (-\delta_0/10, \delta_0/10)^2 \times B_n(0, \delta_0/10) \subset \gamma_t \circ  R^1_t \circ T^1_t(U)$ for all sufficiently small $t > 0$, $\gamma_t\left(R^1_t \circ T^1_t(D)\right) \cap W = \gamma_t(\Omega) \cap W$. Let $\epsilon > 0$.
Define
\begin{align}
    D_t^{\epsilon} &= \left\{(z_1, \dots, z_{n+1}) \in W : \rho \circ \gamma_t^{-1}(z_1, z_2, \dots, z_{n+1}) < 0, \operatorname{Re}z_1 > - d(t)^{{1}/{(1+\epsilon)^2}}\right\}, \text{ and} \\
    \widetilde{D}_t^{\epsilon} &= \left\{z \in D_t^{\epsilon}: \left|h_{\epsilon}(z)\right| > \operatorname{exp}\left(- d(t)^{{1}/{(1+\epsilon)^2}}c_0\right)\right\},
\end{align}
where $c_0 = \operatorname{cos}\left(\frac{\pi}{2(1 + \epsilon)}\right)$
and
$h_\epsilon : D_t^{\epsilon} \to \mathbb{D}$ is defined by 
\begin{equation}
    h_\epsilon(z) = \operatorname{exp}\left(-(-z_1)^{{1}/{(1+ \epsilon)}}\right). 
\end{equation}
For $\epsilon, t > 0$, let
\begin{align}
    d^*(t) &:= 
    \operatorname{min}\{s \in\mathbb{R}^+: se_2 + (-d(t), 0, \dots, 0) \in b\gamma_t(\Omega)\}, \\
    d_1^\epsilon(t) &:=
     \operatorname{sup} \left\{|z'| : z \in W, \phi\left(\left|\frac{z_2}{A(t)} + p_2(t) + \frac{2p_2(t)\phi'(p_2(t)^2)}{A(t)}z_1\right|^2 + \dots + |z_{n+1}|^2\right) \right.\nonumber\\
     &\left. \hspace{55mm} \leq \frac{d(t)^{\frac{1}{(1+\epsilon)}}}{A(t)} 
    + \phi(p_2(t)^2) + \frac{2p_2(t)\phi'(p_2(t)^2)}{A(t)}\operatorname{Re}(z_2)\right\},\\
    d_2^\epsilon(t) &:=
    \operatorname{sup}\left\{|z'| : z \in W, \phi\left(\left|\frac{z_2}{A(t)} + p_2(t) + \frac{2p_2(t)\phi'(p_2(t)^2)}{A(t)}z_1\right|^2 + \dots + |z_{n+1}|^2\right) \right.\nonumber\\
    &\left. \hspace{55mm} \leq \frac{d(t)^{\frac{1}{(1+\epsilon)^2}}}{A(t)} 
    + \phi(p_2(t)^2) + \frac{2p_2(t)\phi'(p_2(t)^2)}{A(t)}\operatorname{Re}(z_2)\right\},
    \end{align}
and $\Sigma : \mathbb{C}^{n+1} \to \mathbb{C}^{n+1}$
by
\begin{equation}
    \Sigma(z_1, z') = \left(\frac{z_1}{d(t)}, \frac{z'}{d^*(t)}\right),
\end{equation}
where $z' = (z_2, \dots, z_{n+1})$.

In this section, we will present our findings using the quantities defined above.

We will prove the limiting behaviour of $\phi, \, d(t), \, d^*(t), d_1^\epsilon(t)$ and $d_2^\epsilon(t)$ at zero in the following lemma.
\begin{lem}\label{3.1}
Let $D \subset \mathbb{C}^{n + 1}$ be as in \eqref{*}. Then
\begin{align*}
    &\lim_{t \to 0^+} \frac{\phi \left(p_2(t)^2\right)}{d(t)} 
    = 0, &&\lim_{t \to 0^+} \frac{\phi'\left(p_2(t)^2\right)}{d(t)} 
    = 0, && \lim_{t \to 0^{+}} \frac{p_2(t)}{d^{*}(t)}  = 0,\\ &\liminf_{t \to 0^+} \frac{d^{*}(t)}{d_1^{\epsilon}(t)} \geq \frac{1}{(1 + \epsilon)^{\frac{1}{2m}}}, \text{ and}  &&\liminf_{{t \to 0^*}}\frac{d^*(t)}{d_2^\epsilon(t)} \geq \frac{1}{(1 + \epsilon)^{\frac{1}{m}}}.
\end{align*}
\end{lem}
\begin{proof}
Without loss of generality, we may assume that $p_2(t) \neq 0$ for small $t > 0$. Since $\Tilde{q}(t) \in C_{\alpha, N}$, for small $t > 0$, $\operatorname{Re}q_1(t) < -\alpha \left({r}(t)\right)^N$.
We have
\begin{equation}
    p(t) = \Tilde{q}(t) + \frac{\nabla \rho(p(t))d(t)}{A(t)},
\end{equation}
which implies
\begin{equation}\label{p_1}
    p_1(t) = \operatorname{Re}q_1(t) + \frac{d(t)}{A(t)},
\end{equation}
and
\begin{equation}\label{p_2}
    p_2(t) = {r}(t) + \frac{2p_2(t)\phi'\left(p_2(t)^2\right)d(t)}{A(t)}.
\end{equation}
From \eqref{p_2}, we get 
\begin{equation}\label{7}
    \lim_{t \to 0^+} \frac{{r}(t)}{p_2(t)} = 1.
\end{equation}
Now from \eqref{p_1}, we have
\begin{align*}
    -&\phi\left(p_2(t)^2\right) = p_1(t) = \operatorname{Re}q_1(t) + \frac{d(t)}{A(t)} < -\alpha \left(r(t)\right)^N + \frac{d(t)}{A(t)}
\end{align*}
and hence
\begin{equation*}
\frac{d(t)}{A(t)} > \alpha \left({r}(t)\right)^N - \phi\left(p_2(t)^2\right)
\end{equation*}
Using \eqref{7}, 
$\alpha \left(r(t)\right)^N - \phi\left(p_2(t)^2\right) \geq 0$
for all sufficiently small $t > 0$. Therefore
\begin{equation*}
    \frac{A(t)}{d(t)} < \frac{1}{\alpha \left({r}(t)\right)^N - \phi\left(p_2(t)^2\right)},
\end{equation*}
which implies 
\begin{equation*}
    \frac{p_2(t)^{(N+1)}}{d(t)} < \frac{1}{A(t)\left(\left(\frac{{r}(t)}{p_2(t)}\right)^{N}\cdot \frac{1}{p_2(t)} - \frac{\phi\left(p_2(t)^2\right)}{p_2(t)^{(N+1)}}\right)},
\end{equation*}
and hence
\begin{equation}\label{8}
    \lim_{t \to 0^+}\frac{p_2(t)^{(N+1)}}{d(t)} = 0.
\end{equation}
So,
\begin{equation}\label{phi}
        \lim_{t \to 0^+} \frac{\phi\left(p_2(t)^2\right)}{d(t)} =         \lim_{t \to 0^+} \frac{\phi\left(p_2(t)^2\right)}{p_2(t)^{(N+1)}}\cdot \frac{p_2(t)^{(N+1)}}{d(t)} = 0.
\end{equation}
Similarly,
\begin{equation}\label{phi'}
\lim_{t \to 0^+} \frac{\phi'\left(p_2(t)^2\right)}{d(t)} = \lim_{t \to 0^+} \frac{\phi'\left(p_2(t)^2\right)}{p_2(t)^{(N+1)}}\cdot \frac{p_2(t)^{(N+1)}}{d(t)} = 0.
\end{equation}
Since $(-d(t), d^*(t), 0, \dots, 0) \in b\gamma_t(\Omega)$,
\begin{equation}\label{9}
    \frac{d(t)}{A(t)} + \frac{2p_2(t)\phi'\left(p_2(t)^2\right)}{A(t)}d^*(t) + \phi(p_2(t)^2) = \phi\left(\left[\frac{d^*(t)}{A(t)} + p_2(t) - \frac{2p_2(t)\phi'(p_2(t)^2)}{A(t)}d(t)\right]^2\right).
\end{equation}
From the above equation, we get $\lim_{t \to 0^{+}} d(t)/d^{*}(t) = 0
$, which implies \begin{equation}\label{10}
    \lim_{t \to 0^{+}} \frac{p_2(t)^{(N + 1)}}{d^{*}(t)} = \lim_{t \to 0^+} \frac{p_2(t)^{N + 1}}{d(t)} \cdot \frac{d(t)}{d^{*}(t)} = 0. 
\end{equation}
From \eqref{9} and \eqref{10}, we get $\lim_{t \to 0^+} {d(t)}/{d^{*}(t)^{N + 1}} = 0$. Hence
\begin{equation}\label{p_2/d^*}
    \lim_{t \to 0^{+}} \frac{p_2(t)}{d^{*}(t)} = \lim_{t \to 0^{+}} \frac{p_2(t)}{d(t)^{\frac{1}{N + 1}}} \cdot  \frac{d(t)^{\frac{1}{N + 1}}}{d^{*}(t)} = 0.
\end{equation}

By using the definitions of $\phi, d_1^{\epsilon}(t),  d_2^{\epsilon}(t)$, and $d^*(t)$, we get
\begin{equation}
\liminf_{t \to 0^{+}} \frac{d^{*}(t)}{d^{\epsilon}_1(t)} \geq \frac{1}{(1 + \epsilon)^{\frac{1}{2m}}}, \, \text{and }
\liminf_{{t \to 0^*}}\frac{d^*(t)}{d_2^\epsilon(t)} \geq \frac{1}{(1 + \epsilon)^{\frac{1}{m}}}.
\end{equation}
This completes the proof.
\end{proof}
In the following lemma, we will prove the localization of extremal integrals.
\begin{lem}\label{localisation lemma 4}
    Let $D \subset \mathbb{C}^{n+1}$ be as in \eqref{*}. Further assume that $D$ is pseudoconvex. For all sufficiently small $t > 0$, let $D_t = \gamma_t\left(R_t^1 \circ T_t^1(D)\right)$. Then, for $\epsilon > 0$, 
    \begin{equation*}
    \lim_{t \to 0^+} {\frac{I_0^{D_t}\left(-d(t), 0\right)}{I_0^{D_t^\epsilon}\left(-d(t), 0\right)}} = 1 \text{ and }\quad
    \lim_{t \to 0^+} {\frac{I_j^{D_t}\left((-d(t), 0), \xi(t)\right)}{I_j^{D_t^\epsilon}\left((-d(t), 0), \xi(t)\right)}} = 1
    \end{equation*}
     for $j = 1,2$, and $\xi(t) \in \mathbb{C}^{n + 1} \setminus \{0\}$.
\end{lem}
\begin{proof}
    Let $\epsilon > 0$. Choose a cut-off function $\chi \in C_c^\infty(B_1(0,1) \times B_{n}(0, 2))$ such that 
\begin{equation*}
    \chi = 1 \text{ on } B_1(0, 1/2) \times B_{n}(0,1), \ \ \ 0 \leq \chi \leq 1. 
\end{equation*}
Define,
\begin{equation*}
    \chi_t(z) = \chi\left(\frac{z_1}{d(t)^{\frac{1}{(1+\epsilon)^2}}}, \frac{z'}{d_1^{\epsilon}(t)}\right) \text{ for } z \in (z_1, z') \in \mathbb{C}^{n + 1}.
\end{equation*}
Let $ z \in \widetilde{D}_t^\epsilon$.
Then, $|z_1| \leq d(t)^{\frac{1}{(1+\epsilon)}}$ and
there exist $t_0(\epsilon) > 0$ such that
\begin{align*}
    \frac{|z_1|}{d(t)^{\frac{1}{(1+\epsilon)^2}}} \leq d(t)^\frac{\epsilon}{(1+\epsilon)^2} < 1/2, \text{  and  }
    \frac{|z'|}{d_1^\epsilon(t)} < 1,
\end{align*}
for each $0<t<t_0(\epsilon)$. Therefore $\chi_t = 1$ on $\widetilde{D}_t^\epsilon$. Let
\begin{equation*}
V_t = W \cap \left\{z \in \mathbb{C}^{n+1} : |\operatorname{Re}z_1| < d(t)^\frac{1}{(1+\epsilon)^2} \right\},
\end{equation*}
where $W$ is defined at the beginning of this section.
Note $\chi_t \in C_c^\infty(V_t)$, and $V_t \cap D_t = D_t^\epsilon$. Let $f \in A^2(D_t^\epsilon)$ and $k \in \mathbb{N}$. Set
\begin{align*}
    \alpha &= \overline{\partial}\left(f\chi_t h_{\epsilon}^k\right) \text{ on } D_t, \text{ and}\\
    \phi(z) &= (2n+8)\operatorname{log}|z - (-d(t),0)|, \text{ for } z \in \mathbb{C}^{n + 1}.   
\end{align*}
By Theorem \ref{Hormander}, we have $u \in L^2_{\operatorname{loc}}(D_t)$ satisfying
\begin{equation*}
    \overline{\partial}u = \alpha \text{ on } D_t,
\end{equation*}and
\begin{equation}\label{Hor inequality 2}
\int_{D_t} \frac{|u|^2e^{-\phi}}{(1+ |z|^2)^2} \, \mathrm{d}V \leq \int_{D_t}|\alpha|^2 e^{-\phi} \, \mathrm{d}V.
\end{equation}
Now,
\begin{align}\label{estimate alpha 2}
    \int_{D_t}|\alpha|^2 e^{-\phi} \, \mathrm{d}V &= \int_{D_t\cap V_t}\frac{|f|^2|\overline{\partial}X_t|^2|h_{\epsilon}|^{2k}}{|z - (-d(t), 0)|^{(2n+8)}} \, \mathrm{d}V \nonumber\\
    &\leq \frac{Ma_t^{2k}}{d(t)^2}\int_{D_t^\epsilon \setminus 
    \widetilde{D}_t^\epsilon} \frac{|f|^2}{|z - (-d(t), 
 0)|^{(2n + 8)}}\, \mathrm{d}V \nonumber\\
    &\leq \frac{M a_t^{2k}}{d(t)^{(2n + 10)}} \int_{D_t^\epsilon \setminus \widetilde{D}_t^\epsilon}|f|^2 \, \mathrm{d}V \nonumber\\
    & \leq \frac{Ma_t^{2k}}{d(t)^{(2n+10)}} \norm{f}_{L^2(D_t^\epsilon)}^2,
\end{align}
where $a_t = \operatorname{exp}\left(-c_0d(t)^{{1}/{(1+\epsilon)^2}}\right)$, and $M > 0$, which may change at each step, depends on both the first-order derivatives of $\chi$ and $\epsilon > 0$.
Now we estimate the left side of \eqref{Hor inequality 2}.
\begin{equation}\label{estimate u_0 2}
   \int_{D_t} \frac{|u|^2e^{-\phi}}{(1+ |z|^2)^2} \, \mathrm{d}V =\int_{D_t} \frac{|u|^2}{(1+|z|^2)^2 \, |z - (-d(t),0)|^{(2n+8)}} \, \mathrm{d}V \geq M_0 \norm{u}_{L^2(D_t)}^2,
\end{equation}
where $M_0 > 0$
depends only on the domain $D$.
From \eqref{Hor inequality 2},
\eqref{estimate alpha 2}, and \eqref{estimate u_0 2}, we have
\begin{equation}\label{res equation}
    \norm{u}_{L^2(D_t)}^2 \leq \frac{M a_t^{2k}}{M_0d(t)^{(2n + 10)}}\norm{f}_{L^2(D_t^\epsilon)}^2, 
\end{equation}
and 
\begin{equation}\label{19}
    \int_{D_t} \frac{|u|^2}{(1+|z|^2)^2 \, |z - (-d(t),0)|^{(2n+8)}} \, \mathrm{d}V \leq \frac{M a_t^{2k}}{d(t)^{(2n + 10)}}.
\end{equation}
From the above equation \eqref{19}, we get
\begin{equation}
        \frac{\partial^{|\alpha| + |\beta|} u}{\partial z^{\alpha} \partial \bar{z}^{\beta}}(-d(t), 0) = 0 \text{ for all multi-indices $\alpha, \beta$ with $|\alpha| + |\beta| \leq 2$.}
\end{equation}
We now apply the above for the $f \in A^2\left(D_t^{\epsilon}\right)$ that satisfies 
\begin{equation*}
    I_0^{D_t^{\epsilon}}(-d(t), 0) = \norm{f}_{L^2(D_t^\epsilon)}^2, \ \ \ f(-d(t), 0) = 1.
\end{equation*}
Set
\begin{equation*}
    g = \frac{\chi f h_{\epsilon}^k - u}{\left(h_{\epsilon}(-d(t),0)\right)^k} \text{ on $D_t$}.
\end{equation*}
It follows that $g$
is holomorphic on $D_t$ and $g(-d(t), 0) = 1$. Therefore
\begin{equation*}
    I_0^{D_t}(-d(t), 0) \leq
    \norm{g}_{L^2(D_t)}^2 \leq \left(\frac{\norm{f}_{L^2(D_t^{\epsilon})} + \norm{u}_{L^2(D_t)}}{\left(h_{\epsilon}(-d(t), 0)\right)^k}\right)^2\leq \left(\frac{1 + \frac{\Tilde{M} a_t^{k}}{d(t)^{(n+5)}}}{\left(h_{\epsilon}(-d(t), 0)\right)^k}\right)^2 \norm{f}_{L^2(D_t^{\epsilon})}^2,
\end{equation*}
where $\Tilde{M} = \sqrt{M/M_0}$. So,
\begin{equation*}
   \frac{I_0^{D_t}\left(-d(t), 0\right)}{I_0^{D_t^\epsilon}\left(-d(t), 0\right)} \leq    \left(\frac{1 + \frac{\Tilde{M} a_t^{k}}{d(t)^{(n + 5)}}}{\left(h_{\epsilon}(-d(t), 0)\right)^k}\right)^2,
\end{equation*}
for each $0 < t < t_0(\epsilon)$.
For every $\epsilon > 0$, there exist $c_{\epsilon} > 0$, such that
\begin{equation*}
    \frac{1}{(1 + \epsilon)^2} < c_{\epsilon} < \frac{1}{(1+\epsilon)}.
\end{equation*}
Choosing $k$ to be the greatest integer of $d(t)^{-c_{\epsilon}}$, we get
\begin{equation}
\lim_{t \to 0^+} {\frac{I_0^{D_t}\left(-d(t), 0\right)}{I_0^{D_t^\epsilon}\left(-d(t), 0\right)}} = 1,
\end{equation}for each $\epsilon > 0$.

Similarly, 
\begin{equation}
    \lim_{t \to 0^+} {\frac{I_j^{D_t}\left((-d(t), 0), \xi(t) \right)}{I_j^{D_t^\epsilon}\left((-d(t), 0), \xi(t) \right)}} = 1 \text{ for all } j = 1,2,
\end{equation}
for each $\epsilon > 0$, and $\xi(t) \in \mathbb{C}^{n + 1} \setminus \{0\}$.
\end{proof}
In the following lemma, we will show that the limiting domain of $f \circ \Sigma (D_t^\epsilon)$ is $\mathbb{D} \times B_n(0,1)$. 
\begin{lem}\label{ScalingLemma2}
    Let $D \subset \mathbb{C}^{n + 1}$ be as in \eqref{*}. Then, for every $\epsilon, \delta > 0,$ there exists $t_0(\delta, \epsilon) > 0$ such that
    \begin{equation}\label{inclusions 2}
        (1-\delta)\mathbb{D}\times B_n(0,1) \subset f \circ \Sigma \left(D_t^\epsilon\right) \subset
        \mathbb{D} \times B_n(0, {d}_2^\epsilon(t)/d^*(t)),
    \end{equation}
for each $ 0 < t < t_0(\delta, \epsilon
)$, where $f(z_1, z') = ((1+z_1)/(1-z_1), z')$.
\end{lem}
\begin{proof}
    Consider $(1-\delta)(z_1, z') \in (1-\delta)\mathbb{D} \times B_n(0,1)$.
    Define,
    \begin{align*}
        w_1 = d(t)\left(\frac{(1-\delta)z_1 -1}{(1-\delta)z_1 + 1}\right), \text{ and}\quad
        w' = (1-\delta) d^*(t)z'.
    \end{align*}
    Then $f \circ \Sigma (w_1, w') = (1-\delta)(z_1, z')$.

To prove the first inclusion of \eqref{inclusions 2}, we show that $(w_1, w') \in D_t^{\epsilon}$.

Since
\begin{align*}
    \operatorname{Re}w_1 = -d(t)\left(\frac{1- (1-\delta)^2|z_1|^2}{\left((1-\delta)x_1 + 1\right)^2 + (1-\delta)^2y_1^2}\right),
\end{align*}
where $z_1 = x_1 + iy_1 \in \mathbb{D}$,
\begin{equation}\label{eq 1*}
        \operatorname{Re}w_1 > \frac{-d(t)}{\delta^2} = - d(t)^{\frac{1}{(1+\epsilon)}}\frac{d(t)^\frac{\epsilon}{(1+ \epsilon)}}{\delta^2}.
\end{equation}
Since $d(t) \to 0$ as $t \to 0^+$, there exists $t_0(\delta, \epsilon)> 0$ such that
\begin{equation}\label{eq 2*}
    d(t)^{\frac{\epsilon}{(1 + \epsilon)}} < \delta^2,
\end{equation}
for each $0<t< t_0(\delta, \epsilon)$. So, from equations \eqref{eq 1*} and \eqref{eq 2*}, we get
\begin{equation*}
    \operatorname{Re}w_1 > -{d(t)^{\frac{1}{(1 + \epsilon)}}} \geq -{d(t)^{\frac{1}{(1 + \epsilon)^2}}}.
\end{equation*}
Also $|w_1| < \delta_0/10$, since $d(t) \to 0$ as $t \to 0^{+}.$

Since $w' = (1-\delta)d^*(t)z'$ and $d^{*}(t) \to 0$ as $t \to 0^{+}$,
\begin{equation*}
    |w'| < d^*(t) < \delta_0/10.
\end{equation*}
We now estimate
\begin{align}\label{rhocircgamma^{-1}}
    \rho \circ \gamma_t^{-1}(w_1,w') &= \frac{\operatorname{Re}w_1}{A(t)} - \frac{2p_2(t)\phi'(p_2(t)^2)}{A(t)}\operatorname{Re}w_2 - \phi(p_2(t)^2)  \nonumber \\
    & \ \ \ \ +\phi\left(\left|\frac{w_2}{A(t)} + p_2(t) + \frac{2p_2(t)\phi'(p_2(t)^2)}{A(t)}w_1\right|^2 + |w_3|^2 + \dots + |w_{n+1}|^2\right)\nonumber\\
    &\leq -\frac{d(t)}{A(t)}\left(\frac{1 - (1-\delta)^2}{(2-\delta)^2 + (1-\delta)^2}\right)
    + \frac{2p_2(t) \phi'(p_2(t)^2)}{A(t)}(1-\delta)d^*(t) - \phi(p_2(t)^2) \nonumber \\
    & \ \ \ \ + \phi\left(\left(\frac{|w_2|}{A(t)} + p_2(t) + \frac{2p_2(t)\phi'(p_2(t)^2)}{A(t)}|w_1|\right)^2 + |w_3|^2 + \dots + |w_{n+1}|^2\right)\nonumber\\
    &\leq -\frac{d(t)}{A(t)}\left(\frac{1 - (1-\delta)^2}{(2-\delta)^2 + (1-\delta)^2}\right)
    + \frac{2p_2(t) \phi'(p_2(t)^2)}{A(t)}(1-\delta)d^*(t) - \phi(p_2(t)^2) \nonumber\\
    & \,\, \, \  \ \ + 
    \phi\Bigg[\left((1-\delta)d^*(t)|z_2| + p_2(t) + \frac{4p_2(t)\phi'\left(p_2(t)^2\right)d(t)}{\delta^2}\right)^2  \nonumber\\
    & \hspace{58mm} +(1-\delta)^2d^*(t)^2\left(|z_3|^2 + \dots + |z_{n+1}|^2\right)\Bigg].
\end{align}
Since $(-d(t), d^*(t), 0, \dots, 0) \in b\gamma_t(\Omega)$,
\begin{equation}\label{rel in d^*}
    \frac{d(t)}{A(t)} + \frac{2p_2(t)\phi'(p_2(t)^2)}{A(t)}d^*(t) + \phi(p_2(t)^2) = \phi\left(\left[\frac{d^*(t)}{A(t)} + p_2(t) - \frac{2p_2(t)\phi'(p_2(t)^2)}{A(t)}d(t)\right]^2\right).
\end{equation}
From Lemma \ref{3.1}, $\lim_{t \to 0^+}p_2(t)/d^*(t) = 0$. Therefore
\begin{align}  
    \lim_{t \to 0^+}&\frac{\left((1-\delta)d^*(t)|z_2| + p_2(t) + \dfrac{4p_2(t)\phi'(p_2(t)^2)d(t)}{\delta^2}\right)^2 +  (1-\delta)^2d^*(t)^2(|z_3|^2 + \dots + |z_{n+1}|^2)}{\left(\frac{d^*(t)}{A(t)} + p_2(t) - \frac{2p_2(t)\phi'(p_2(t)^2)}{A(t)}d(t)\right)^2}\nonumber \\
    &= (1-\delta)^2|z'|^2 <(1-\delta)^2.
\end{align}
Since
\begin{equation}
        \lim_{r \to 0^{+}}\frac{\phi(rx)}{\phi(r)} = \begin{cases} 
      0, &\text{if } 0 < x < 1, \text{ and} \\
      \infty, &\text{if }  x > 1,
      \end{cases}
\end{equation}
we take $t_0(\delta, \epsilon) > 0$ sufficiently small such that
\begin{align*}
    &\phi\left(\left((1-\delta)d^*(t)|z_2| + p_2(t) + \frac{4p_2(t)\phi'(p_2(t)^2)d(t)}{\delta^2}\right)^2 +  (1-\delta)^2d^*(t)^2(|z_3|^2 + \dots + |z_{n+1}|^2)\right)\\
    & \ \ \ \ \ \ \ \ \ \ \ \ \ \ \ \ \leq \frac{1}{2}\frac{1 - (1-\delta)^2}{(2-\delta)^2 + (1-\delta)^2}\phi\left(\left[\frac{d^*(t)}{A(t)} + p_2(t) - \frac{2p_2(t)\phi'(p_2(t)^2)}{A(t)}d(t)\right]^2\right) \\
    & \ \ \ \ \ \ \ \ \ \ \ \ \ \ \ \ = \frac{1}{2}\frac{1 - (1-\delta)^2}{(2-\delta)^2 + (1-\delta)^2}\left( \frac{d(t)}{A(t)} + \frac{2p_2(t)\phi'(p_2(t)^2)}{A(t)}d^*(t) + \phi(p_2(t)^2)\right) \ \ \ \text{(by \eqref{rel in d^*})},
\end{align*}
for each $0 < t < t_{0}(\delta, \epsilon).$
From \eqref{rhocircgamma^{-1}}, we have
\begin{align*}
    \rho \circ \gamma_t^{-1}(w_1, w') &\leq -\frac{d(t)}{2A(t)} \frac{1 - (1-\delta)^2}{(2-\delta)^2 + (1-\delta)^2}\\ &\ \ \ \ + \frac{2p_2(t)\phi'(p_2(t)^2)}{A(t)}d^*(t)\left((1-\delta) + \frac{1}{2}\frac{1 - (1-\delta)^2}{(2-\delta)^2 + (1-\delta)^2} \right) \\
    & \ \ \ \ + \phi(p_2(t)^2)\left(-1 + \frac{1}{2}\frac{1 - (1-\delta)^2}{(2-\delta)^2 + (1-\delta)^2}\right).
\end{align*}
Using lemma \ref{3.1} again, we see that
$\phi(p_2(t)^2)/d(t) \to 0$ and $\phi'(p_2(t)^2)/d(t) \to 0$ as $t \to 0^{+}.$
Hence, we can take $t_0(\delta, \epsilon) > 0$ sufficiently small so that
\begin{equation}
    \rho \circ \gamma_t^{-1}(w_1, w') < 0,
\end{equation}
for each $0 < t < t_0(\delta, \epsilon)$. This implies $(w, w') \in D_t^\epsilon$ and hence,
\begin{equation}
    (1-\delta)\mathbb{D}\times B_n(0,1) \subset f \circ \Sigma\left(D_t^\epsilon\right),
\end{equation}for each $0 < t < t_0(\delta, \epsilon)$.

Conversely, assume $(u_1, u') \in f \circ \Sigma\left(D_t^\epsilon\right)$. Then, 
\begin{equation}
    (u_1, u') = f \circ \Sigma (w_1, w'),
\end{equation}
for some $(w_1, w') \in D_t^\epsilon$. 

Since $\phi$ is convex,
$\operatorname{Re}w_1 < 0$, and hence $u_1 \in \mathbb{D}$. So $(u_1, u') \in \mathbb{D} \times B_n(0, d_2^\epsilon(t)/d^*(t))$.
\end{proof}

\begin{lem}\label{I_1 2}
     Let $D \subset \mathbb{C}^{n+1}$ be as in \eqref{*}. Further, assume that $D$ is pseudoconvex and $q$ is an $(\alpha, N)$-cone type stream approaching $0$. Then
     \begin{align*}
         &\lim_{t \to 0^+} \frac{(2d(t))^2(d^{*}(t))^{2n} I_0^{\mathbb{D} \times B_n(0, 1)}(0)}{I_0^{D_t}((-d(t), 0))} = 1  \quad\text{and } \\
         &\lim_{t \to 0^+} \frac{(2d(t))^2(d^{*}(t))^{2n} I_j^{\mathbb{D} \times B_n(0, 1)}(0, (f \circ \Sigma)'(-d(t), 0) \xi(t))}{I_j^{D_t}((-d(t), 0), \xi(t))} = 1     
     \end{align*}
     for $j = 1, 2$, and
     $\xi(t) \in \mathbb{C}^{n + 1} \setminus \{0\}$. Recall that $D_t = \gamma_t\left(R_t^1 \circ T_t^1(D)\right)$.
\end{lem}
\begin{proof}
 Let $t > 0$ be small, and let  $\xi(t) \in \mathbb{C}^{n + 1} \setminus \{0\}$. 
 For $j = 1,2$, we have
    \begin{equation}\label{24}
        I_j^{D_t^{\epsilon}}((-d(t), 0), \xi(t)) = {(2d(t))^2d^*(t)^{2n} I_j^{f \circ \Sigma (D_t^{\epsilon})}(0, (f \circ \Sigma)'(-d(t), 0)\xi(t))}.
    \end{equation}
    Let $\epsilon, \delta > 0$. From Lemma \ref{ScalingLemma2}, there exists $t_{0}(\delta, \epsilon) > 0$ such that
    \begin{align*}
        I_j^{(f\circ \Sigma(D_t^{\epsilon}))}(0, (f \circ \Sigma)'(-d(t),0)\xi(t)) &\geq I_j^{(1-\delta)\mathbb{D} \times B_n(0,1)}(0, (f \circ \Sigma)'(-d(t),0)\xi(t))\\
        &= (1 - \delta)^{2(n + 1)} I_j^{\mathbb{D} \times B_n(0,1)}(0, (f \circ \Sigma)'(-d(t),0)\xi(t)),
    \end{align*}
    which implies
    \begin{equation}\label{136}
        \frac{I_j^{\mathbb{D} \times B_n(0,1)}(0, (f \circ \Sigma)'(-d(t),0)\xi(t))}{I_j^{(f\circ \Sigma(D_t^{\epsilon}))}(0, (f \circ \Sigma)'(-d(t),0)\xi(t))} \leq \frac{1}{(1 - \delta)^{2(n + 1)}},
    \end{equation}
    for each $0 < t < t_0(\delta, \epsilon)$.

    Again using Lemma \ref{ScalingLemma2}, we have
    \begin{align*}
        I_j^{(f\circ \Sigma(D_t^{\epsilon}))}(0, (f \circ \Sigma)'(-d(t),0)\xi(t)) &\leq I_j^{(\mathbb{D} \times B_n(0, d_2^{\epsilon}(t)/d^{*}(t)))}{(0, (f \circ \Sigma)'(-d(t),0)\xi(t))}\\
        &= \left(\frac{d_2^{\epsilon}(t)}{d^{*}(t)}\right)^{2n}I_j^{(\mathbb{D} \times B_n(0,1))}{(0, (f \circ \Sigma)'(-d(t),0)\xi(t))},
    \end{align*}
    which implies
    \begin{equation}\label{138}
        \frac{I_j^{(\mathbb{D} \times B_n(0,1))}{(0, (f \circ \Sigma)'(-d(t),0)\xi(t))}}{I_j^{(f\circ \Sigma(D_t^{\epsilon}))}(0, (f \circ \Sigma)'(-d(t),0)\xi(t))} \geq \left(\frac{d^*(t)}{d_2^{\epsilon}(t)}\right)^{2n},
    \end{equation}
    for each $0 < t < t_0(\delta, \epsilon)$.

    Now 
    \begin{align}\label{139}
        &\frac{(2d(t))^2 d^*(t)^{2n}I_j^{\mathbb{D} \times B_n(0,1)}(0, (f \circ \Sigma)'(-d(t),0)\xi(t))}{{I_j^{D_t}((-d(t),0), \xi(t))}}\nonumber \\
        &\,\,\,\,\,\,\,\,\,=\frac{I_j^{\mathbb{D} \times B_n(0,1)}(0, (f \circ \Sigma)'(-d(t),0)\xi(t))}{{{I_j^{(f\circ \Sigma(D_t^{\epsilon}))}(0, (f \circ \Sigma)'(-d(t),0)\xi(t))}}} \cdot \frac{{{I_j^{(f\circ \Sigma(D_t^{\epsilon}))}(0, (f \circ \Sigma)'(-d(t),0)\xi(t))}}}{(2d(t))^{-2}d^*(t)^{-2n}I_j^{D_t}((-d(t),0), \xi(t))}\nonumber\\
        &\,\,\,\,\,\,\,\,\,=\frac{I_j^{\mathbb{D} \times B_n(0,1)}(0, (f \circ \Sigma)'(-d(t),0)\xi(t))}{{{I_j^{(f\circ \Sigma(D_t^{\epsilon}))}(0, (f \circ \Sigma)'(-d(t),0)\xi(t))}}} \cdot \frac{{{I_j^{D_t^{\epsilon}}((-d(t), 0), \xi(t))}}}{I_j^{D_t}((-d(t),0), \xi(t))} \, \, \, \, \, \, \text{     
        (from \eqref{24})}.
    \end{align}
    By using \eqref{136}, \eqref{139}, and Lemma \ref{localisation lemma 4}, we get
    \begin{equation}
        \limsup_{t \to 0^+}\frac{(2d(t))^2 d^*(t)^{2n}I_j^{\mathbb{D} \times B_n(0,1)}(0, (f \circ \Sigma)'(-d(t),0)\xi(t))}{{I_j^{D_t}((-d(t),0), \xi(t))}} \leq \frac{1}{(1 - \delta)^{2(n + 1)}} ,
    \end{equation}
    for each $\delta > 0$. Hence
    \begin{equation}\label{29}
        \limsup_{t \to 0^+}\frac{(2d(t))^2 d^*(t)^{2n}I_j^{\mathbb{D} \times B_n(0,1)}(0, (f \circ \Sigma)'(-d(t),0)\xi(t))}{{I_j^{D_t}((-d(t),0), \xi(t))}}  \leq {1}.
    \end{equation}
    By using \eqref{138}, \eqref{139}, Lemma \ref{3.1}, and Lemma \ref{localisation lemma 4}, we get
    \begin{equation*}
        \liminf_{t \to 0^{+}} \frac{(2d(t))^2 d^*(t)^{2n}I_j^{\mathbb{D} \times B_n(0,1)}(0, (f \circ \Sigma)'(-d(t),0)\xi(t))}{{I_j^{D_t}((-d(t),0), \xi(t))}}  \geq \frac{1}{(1 + \epsilon)^{2n/m}},
    \end{equation*}
    for $\epsilon > 0$. Hence
    \begin{equation}\label{31}
        \liminf_{t \to 0^{+}} \frac{(2d(t))^2 d^*(t)^{2n}I_j^{\mathbb{D} \times B_n(0,1)}(0, (f \circ \Sigma)'(-d(t),0)\xi(t))}{{I_j^{D_t}((-d(t),0), \xi(t))}}  \geq 1.
    \end{equation}
    From \eqref{29} and \eqref{31}, we have
    \begin{equation}
         \lim_{t \to 0^+} \frac{(2d(t))^2 d^*(t)^{2n}I_j^{\mathbb{D} \times B_n(0,1)}(0, (f \circ \Sigma)'(-d(t),0)\xi(t))}{{I_j^{D_t}((-d(t),0), \xi(t))}} = 1.
    \end{equation}
   Similarly, we have
   \begin{equation}
    \lim_{t \to 0^+} \frac{(2d(t))^2(d^{*}(t))^{2n} I_0^{\mathbb{D} \times B_n(0, 1)}(0)}{I_0^{D_t}(-d(t), 0)} = 1.   
   \end{equation}
\end{proof}
We now present the proof of the Theorem \ref{1.3}.

\begin{proof}[Proof of Theorem \ref{1.3}]
Let $\xi \in \mathbb{C}^{n + 1} \setminus \{0\}$. We use $\xi_{N,t}$ and $\xi_{T,t}$ to denote the complex normal and complex tangential components of $\xi$ with respect to $\pi(q(t))$, i.e., $\xi = \xi_{N,t} + \xi_{T,t}$, where $\xi_{T,t} \in T_{\pi(q(t))}^{\mathbb{C}}(bD)$ and $\xi_{N,t} \perp T_{\pi(q(t))}^{\mathbb{C}}(bD)$.
We then have
\begin{equation*}
     \xi_{N, t} = \langle R^2_t \circ R^1_t(\xi), e_1\rangle (R_t^2 \circ R_t^1)^{-1
     }(e_1), \, \text{and } \xi_{T,t} = \sum_{j = 2}^{n + 1}\langle R^2_t \circ R^1_t(\xi), e_j\rangle (R_t^2 \circ R_t^1)^{-1
     }(e_j).
\end{equation*}
Then $\xi = \xi_{N, t} + \xi_{T,t},
|\xi_{N, t}| = |\langle R^2_t \circ R^1_t(\xi), e_1 \rangle|$, and 
$|\xi_{T,t}| = \sqrt{\sum_{j = 2}^{n + 1}|\langle R^2_t \circ R^1_t(\xi), e_j\rangle|^2}$.
Here
 \begin{align*}
 \kappa_D(q(t)) &= \kappa_{D_t}(-d(t), 0),\\
 B_D(q(t); \xi) &= B_{D_t}((-d(t), 0);R_t^2 \circ R^1_t(\xi)), \, \text{and} \\
 H_D(q(t); \xi) &= H_{D_t}((-d(t), 0); R_t^2 \circ R_t^1(\xi)).
 \end{align*}
By using the Proposition \ref{Fuchs}, and Lemma \ref{I_1 2}, we get
\begin{align}
    \lim_{t \to 0^+} \frac{\kappa_D(q(t)) (2d(t))^2(d^{*}(t))^{2n}}{K_{\mathbb{D} \times B_n(0, 1)}(0) 
 }  &= 1\\
 \lim_{t \to 0^{+}} \frac{B_D(q(t); \xi)}{B_{\mathbb{D} \times B_n(0, 1)}(0; (f \circ \Sigma)'(-d(t), 0) R_t^2 \circ R_t^1(\xi))} &= 1, \, \text{and}\\
 \lim_{t \to 0^{+}} \frac{H_D(q(t); \xi) - 2}{H_{\mathbb{D} \times B_n(0, 1)}(0; (f \circ \Sigma)'(-d(t), 0) R_t^2 \circ R_t^1(\xi)) - 2} &= 1.
\end{align}
Since,
\begin{equation*}
    (f \circ \Sigma)'(-d(t), 0) R_t^2 \circ R_t^1(\xi) = \langle R_t^2 \circ R_t^1 (\xi), e_1\rangle \frac{e_1}{2d(t)} + \sum_{j = 2}^{n + 1}\langle R_t^2 \circ R_t^1 (\xi), e_j\rangle \frac{e_j}{d^{*}(t)}.
\end{equation*}
We get the theorem by using Lemma \ref{Formula for kernel}.
\end{proof}
\section{An Example}\label{counterexample}
Some of our results (for $n = 1$) are stated in Kim and Lee \cite{Sunhong 2002} but their proofs rely crucially on \cite{Sunhong 2002}*{Proposition 2} which is not true.
In this section, we provide a concrete example
to demonstrate that.

Suppose $D \subset \mathbb{C}^2$ and $U \subset \mathbb{C}^2$ are defined as in \eqref{*} with $\phi(x) = e^{-1/x}$.
Let $q(t) = (-t, 0)$ for all sufficiently small $t > 0$, define
\begin{equation*}
    \Omega_t = \left\{(w,z) \in U: \rho(w, z) < 0, \operatorname{Re}w > -t^{1/2}\right\} \subset D \cap U.
\end{equation*}
For small $t > 0$, choose $d_1(t)>0$ and $d^*(t) > 0$ such that $\left(-{t}^{1/2}, d_1(t)\right), \left(-t, d^*(t)\right) \in bD \cap U$, which implies
\begin{equation*}
     -{t^{1/2}} + \operatorname{exp}\left(\frac{-1}{d_1^2(t)}\right) = 0 \text{  and  }  -t + \operatorname{exp}\left(\frac{-1}{(d^*(t))^2}\right) = 0.
\end{equation*}
Let $\Sigma(w, z) = \left(w/t, z/d^*(t)\right)$. Using the above two equations, we get
\begin{equation}
    \frac{d_1(t)}{d^*(t)} = \sqrt{2},
\end{equation}
for small $t > 0$. 

In \cite{Sunhong 2002}*{Proposition 2},
it is claimed that the limit domain 
$\Sigma\left(\Omega_t\right)$ is $\mathbb{H} \times \mathbb{D}$. But we can easily see that, $\Sigma\left(-t^{1/2}, d_1(t)\right) = \left(-1/{t^{1/2}}, \sqrt{2}\right)$ which implies the limit domain of $\Sigma(\Omega_t)$ is not $\mathbb{H} \times \mathbb{D}$. 

The source of this error in Kim and Lee \cite{Sunhong 2002} is due to the following incorrect claim (see \cite{Sunhong 2002}*{equation (17)}).
\begin{equation}\label{voilate eq}
    \lim_{t \to 0^+} \frac{\phi\left(\left|d^*(t)u_2\right|^2\right)}{\sqrt{\phi\left(d^*(t)^2\right)}} = \begin{cases}
     0, & \text{if }|u_2|\leq 1 , \text{ and}\\
     \infty, & \text{if }|u_2| > 1,
     \end{cases}
\end{equation}
which is not true in general. Again let $\phi(x) = e^{-1/x}$, and $d^*(t)$ as defined above. Now
\begin{align*}
        \frac{\phi\left(|d^*(t)u_2|^2\right)}{\sqrt{\phi\left(d^*(t)^2\right)}} &= \frac{\operatorname{exp}\left({-1}/{|d^*(t)u_2|^2}\right)}{\operatorname{exp}\left({-1}/{2|d^*(t)|^2}\right)}\\
        &= \operatorname{exp}\left(\frac{-1}{|d^*(t)|^2}\frac{(2-|u_2|^2)}{2|u_2|^2}\right).
\end{align*}
As $t \rightarrow 0^+$, we get $d^*(t) \rightarrow 0$. Therefore,
\begin{equation}
     \lim_{t \to 0^+}\frac{\phi\left(|d^*(t)u_2|^2\right)}{\sqrt{\phi\left(d^*(t)^2\right)}} = \begin{cases}
     0, & \text{if }|u_2|< \sqrt{2} , \text{ and}\\
     \infty, & \text{if }|u_2|> \sqrt{2},
     \end{cases}
\end{equation}
which contradicts \eqref{voilate eq}.
\section*{Acknowledgements}
I would like to thank my advisor Sivaguru Ravisankar for several fruitful discussions, suggestions and comments.

\def\MR#1{\relax\ifhmode\unskip\spacefactor3000 \space\fi%
  \href{http://www.ams.org/mathscinet-getitem?mr=#1}{MR#1}}
 
\end{document}